\documentclass{amsart}
\usepackage{amsmath, amsthm, amssymb, amsfonts, color}
\usepackage{float,subfig}
\usepackage{graphicx,epstopdf}
\newtheorem{theorem}{Theorem}[section]
\newtheorem{corollary}{Corollary}

\newtheorem{lemma}[theorem]{Lemma}
\newtheorem{proposition}{Proposition}

\theoremstyle{definition}
\newtheorem{definition}[theorem]{Definition}

\usepackage{hyperref}
\usepackage{geometry}
\newcommand{\ip}[2]{\left \langle #1,\ #2\right \rangle}

\usepackage{appendix}
\usepackage{graphicx}
\usepackage{lineno}
\usepackage{array}
\usepackage{longtable}

\usepackage{url}
\usepackage{soul}

\author{Benjamin P. Russo$^{1*}$}
\thanks{$^{1}$ Farmingdale State College (SUNY), Department of Mathematics, \href{}{russobp@farmingdale.edu}}
\author{Joel A. Rosenfeld$^{4\dagger\ddagger}$}
\thanks{$^4$ University of South Florida, Department of Mathematics and Statistics, \href{}{rosenfeldj@usf.edu}}
\thanks{$^\dagger$ Research supported by AFOSR Award FA9550-20-1-0127.}
\thanks{$^\ddagger$ Research supported by NSF grant ECCS-2027976.}
\thanks{$^*$ Corresponding author}
\title{Liouville Operators over the Hardy Space}
\date{}
\begin{document}
\maketitle

\begin{abstract}The role of Liouville operators in the study of dynamical systems through the use of occupation measures have been an active area of research in control theory over the past decade. This manuscript investigates Liouville operators over the Hardy space, which encode complex ordinary differential equations in an operator over a reproducing kernel Hilbert space.
\end{abstract}

\section{Introduction}
\label{intro}

Traditionally the study of Liouville operators has been constrained to Banach spaces of continuous functions, where moment problems relating Liouville operators and occupation measures have been investigated. It was observed in \cite{rosenfeld2019occupation} that the functional relationship between occupation measures and Liouville operators can be fruitfully exploited within the context of reproducing kernel Hilbert spaces, which gave rise to occupation kernels (cf. \cite{rosenfeld2019occupation, rosenfeldCDC2019}).

This shift in the study of Liouville operators have led to several nontrival results in system identification and dynamic mode decomposition (DMD) \cite{rosenfeld2021dynamic,rosenfelddmd}. DMD has traditionally invoked Liouville operators as continuous generators for Koopman (or composition) semi-groups corresponding to discrete time dynamics \cite{kutz2016dynamic}. Work such as \cite{williams2015kernel}, connected the study of DMD and Koopman operators with RKHSs, where continuous time dynamics were discretized and the discretized system was analyzed as a proxy for the continuous time system. This discretization has thus far been necessary for the application of DMD to a dynamical system. Through the incorporation of occupation kernels, \cite{rosenfelddmd} gave a method of the direct DMD analysis of continuous time systems through the Liouville operator over a RKHS. It should be emphasized that the collection of Liouville operators is strictly larger than that of Koopman generators, where the latter requires that a dynamical system admits a discretization. This is not always possible, since dynamics such as $\dot x = 1+x^2$ are not discretizable, yielding a solution with finite escape time.

These recent results position the study of Liouville operators over RKHSs as an important research direction for both pure and applied mathematics. Important for both is the characterization of densely defined Liouville operators over various RKHSs, where each new characterization opens new relations for the data driven study of continuous dynamical systems. The resolution of these questions inform the selection of RKHS for particular applications in systems theory. Moreover, the introduction of \emph{scaled} Liouville operators (cf. \cite{rosenfelddmd}) allows for the representation of a dynamical system through a compact operator over the exponential dot product kernels' native space (the real valued counterpart of the Fock space). The idea of scaled Liouville operators are expanded here as Liouville weighted composition operators, which combines a composition operator with the Liouville operator to produce a bounded and sometimes compact operator that represents a dynamical system.

This manuscript investigates Liouville and Liouville weighted composition operators over the Hardy space. The Hardy space provides a model for the broader investigation of these operators, where properties such as the inner outer factorization of Hardy space functions \cite{hoffman2007banach}, the characterization of densely defined multiplication operators \cite{sarason2008unboundedtoeplitz}, and the representation of Hardy space functions as both analytic functions within the disc and their representation as a subspace of $L^2$ of the circle allow provide tools through which Liouville operators, their spectrum, and symbols may be investigated.

\section{Definitions and Preliminaries}
\begin{definition}A reproducing kernel Hilbert space over a set $X$ is a Hilbert space of functions in which point evaluation $e_x(f)=f(x)$ is a continuous linear functional for all $x\in X$. Thus Riesz representation guarantees that for each $x\in X$ there exists a function $k_x\in H$ such that $f(x)=\ip{f}{k_x}$.
\end{definition}
One of the most studied reproducing kernel Hilbert spaces is the complex Hardy-Hilbert space over the disc. The Hardy space over the disc consists of analytic continuations of $L^2$ functions of the circle $\mathbb{T}$ whose Fourier coefficients, $\hat f(n) := \int_{-\pi}^\pi f(e^{i\theta}) e^{-i n \theta} d\theta$, are non-zero only for $n\geq 0$. In fact, one can define \[H^2(\mathbb{T})=\left\{f\in L^2(\mathbb{T}): f(z)=\sum_{n=0}^\infty \hat{f}(n)z^n\right\}.\] Functions in the Hardy space over the \emph{disc} can then be obtained by taking functions in $H^2(\mathbb{T})$ and integrating them against the Poisson kernel \cite{katznelson}. This amounts to defining $f(z)=\sum_{n=0}^\infty f_nz^n$ for $z\in \mathbb{D}$ where the Taylor coefficients agree with the Fourier coefficients of the function on the circle, $f_n=\hat{f}(n)$. Thus, 
\[H^2(\mathbb{D})=\left\{f:\mathbb{D}\rightarrow\mathbb{C}: f(z)=\sum_{n=0}^\infty f_nz^n \text{ and }\sum_{n=0}^\infty |f_n|^2<\infty \right\},\]
where radial limits exists almost everywhere. Of vast importance in Hardy space theory is the inner-outer factorization theorem \cite{cimaross, rosenblumrovnyak}. 
\begin{theorem}
Every function $f\in H^2$ admits a representation as $f=F\varphi$ where $F$ is an inner function and $\varphi$ is a outer function. This factorization is unique up to a unimodular constant.
\end{theorem} 
This representation is called an \emph{inner-outer factorization} for the function $f\in H^2$. An inner function is a function $F\in H^2$ such that $|F|\leq 1$ and an outer function $G$ is a function such that $GH^2$ is dense in $H^2$. Outer functions have representation as exponentials \cite{vukotic}, i.e. if $G$ is outer then 
\[G(z)=\exp\left (\frac{1}{2\pi}\int_0^{2\pi}\frac{e^{i\theta}+z}{e^{i\theta}-z}\log|\tilde{G}(e^{i\theta})|d\theta\right)\]
for some function $\tilde{G}$. This inner outer factorization is used in establishing many properties for functions in $H^2$. As an example, bounded multiplication operators on $H^2(\mathbb{D})$ have been extensively studied, where it has been determined that the bounded multipliers for $H^2(\mathbb{D})$ is the collection of bounded analytic functions on the disc, $H^\infty(\mathbb{D})$. In contrast, densely defined multiplication operators have received less attention but still have a complete description due to Sarason utilizing the inner outer factorization that exists for functions in $H^2$ \cite{sarason2008unboundedtoeplitz}.

\section{Liouville, Liouville Weighted Composition Operators, and Occupation Kernels}

Given a function $f:\mathbb{D} \to \mathbb{C}$, and setting $\mathcal{D}(A_f) := \{ g \in H^2 : f(\cdot) \frac{d}{dz} g(\cdot) \in H^2\}$, the Liouville operator with symbol $f$, $A_f : \mathcal{D}(A_f) \to H^2$, is given as $A_f g = f(\cdot) \frac{d}{dz} g(\cdot)$. Liouville operators are automatically closed with this domain (cf. \cite{rosenfeld2019occupation}), and thus, when $\mathcal{D}(A_f)$ is all of $H^2$, $A_f$ is bounded. $A_f$ is nearly always unbounded, owing to the inclusion of the differentiation operator, with the notable exception of $f \equiv 0$. When $A_f$ is densely defined, it possesses a well defined adjoint \cite{pedersen2012analysis}.

Let $T > 0$ and suppose that $\theta : [0,T] \to \mathbb{D}$ defines a continuous signal in $\mathbb{D}$. The functional on $H^2$ given as $g \mapsto \int_0^T g(\theta(t)) dt$ is bounded, and hence, there is a function in $H^2$, denoted $\Gamma_{\theta}$, such that $\int_0^T g(\theta(t)) dt = \langle g, \Gamma_{\theta} \rangle_{H^2}$ which is called the occupation kernel corresponding to $\theta$ in $H^2$. When $\gamma: [0,T] \to \mathbb{D}$ is a trajectory satisfying $\dot \gamma = f(\gamma)$, then $\Gamma_{\gamma} \in \mathcal{D}(A_f^*)$ and
\begin{equation}\label{eq:occupationkernelrelation}
    A_f^* \Gamma_\gamma = K_{\gamma(T)} - K_{\gamma(0)},
\end{equation} which is a critical relation in the development of finite rank represenations of Liouville operators in \cite{rosenfelddmd}.

Clearly, as Liouville operators are modally unbounded, a sequence of finite rank representations is not expected to converge to the Liouville operator itself. A remedy for this shortcoming was the introduction of scaled Liouville operators, $A_{f,a} g = a f(\cdot) \frac{d}{dz} g(a \cdot)$ for $0 < a < 1$, which are compact operators for a wide range of $f$ (e.g. $A_{f,a}$ is compact when $f$ is a polynomial) (cf. \cite{rosenfelddmd}). Though this was demonstrated for the exponential dot product space, the same proof holds over the Hardy space. Scaled Liouville operators can be seen as a special case of \textit{Liouville weighted composition operators} introduced in this manuscript as $A_{f,\varphi} g := f \frac{d}{dz} g(\varphi(\cdot)) \frac{d}{dz} \varphi(\cdot)$.

During the preparation of this manuscript, the authors became aware of a recently published work \cite{fatehi2021normality}, which gives a similar operator, $W_{f,\varphi} g = f(\cdot)\frac{d}{dz} g(\varphi(\cdot))$, and which contains Liouville weighted composition operators as a proper subset. However, with the generality introduced in \cite{fatehi2021normality}, there is a loss of structure of the operators that is exploited here. Specifically, note that for the same trajectory, $\gamma$, given above, the relation \eqref{eq:occupationkernelrelation} extends to Liouville weighted composition operators as follows:

\begin{theorem}\label{thm:occupationkernel_weightedcomp}
Suppose that $A_{f,\varphi}:\mathcal{D}(A_{f,\varphi}) \to H^2$ is a closed densely defined operator with domain $\mathcal{D}(A_{f,\varphi}) := \{ g \in H^2 : f(\cdot) \frac{d}{dz}g(\varphi(\cdot))\frac{d}{dz}\varphi(\cdot)\}$, and suppose that $\gamma:[0,T]\to H^2$ satisfies $\dot \gamma = f(\gamma)$. Then $\Gamma_{\gamma} \in \mathcal{D}(A_{f,\varphi}^*)$ and
\begin{equation}\label{eq:weightedcomprelation}
A_{f,\varphi}^*\Gamma_{\gamma} = K_{\varphi(\gamma(T))} - K_{\varphi(\gamma(0))}.
\end{equation}
\end{theorem}

\begin{proof}
Let $g \in \mathcal{D}(A_{f,\varphi})$, then
\begin{gather*}
    \langle A_{f,\varphi} g, \Gamma_{\gamma}\rangle_{H^2} = \int_{0}^T \frac{d}{dz} g(\varphi(\gamma(t))) \frac{d}{dz}\varphi(\gamma(t)) f(\gamma(t))\\
    = \int_0^T \dot g(\varphi(\gamma(t))) dt = g(\varphi(\gamma(T))) - g(\varphi(\gamma(0)))\\
    = \langle g, K_{\varphi(\gamma(T))} - K_{\varphi(\gamma(0))}\rangle_{H^2}.
\end{gather*}

Hence, the functional $g \mapsto \langle A_{f,\varphi} g, \Gamma_{\gamma}\rangle_{H^2}$ is bounded by an application of the Cauchy-Schwarz inequality and the theorem follows.
\end{proof}

To wit, the inclusion of $\frac{d}{dz}\varphi$ in the multiplication symbol of the operator allows for the exchange of $\frac{d}{dz} g(\varphi(\gamma(t))) \frac{d}{dz}\varphi(\gamma(t)) f(\gamma(t))$ with $\dot g(\varphi(\gamma(t))$. This replacement yields the adjoint relation in \eqref{eq:weightedcomprelation}. 

It should be noted that for any continuous signal, $\theta : [0,T] \to \mathbb{D}$, the occupation kernel corresponding to $\theta$ is contained in the domain of the adjoint of any densely defined Liouville weighted composition operator through a different argument. This follows from expressing the occupation kernel as
\begin{gather*} \Gamma_{\theta}(z) = \langle \Gamma_{\theta}, K_z \rangle_{H^2(\mathbb{D})} = \overline{\langle K_z, \Gamma_{\theta} \rangle_{H^2(\mathbb{D})}}= \int_{0}^T \overline{K_z(\theta(t))} dt = \sum_{n=0}^\infty \left( \int_0^T \overline{\theta(t)}^n dt \right) z^n.
\end{gather*}

\section{Densely Defined Liouville Operators over the Hardy Space}

As polynomials are dense inside of $H^2(\mathbb{D})$, the existence of densely defined Liouville operators over the Hardy space follows immediately. In particular, if $f$ is a polynomial, $A_f p(z)$ is a polynomial for every polynomial $p(z)$. The establishment of a broader class of densely defined Liouville operators can be obtained by following the classification of densely defined multiplication operators over the Hardy space, as was done in \cite{sarason2008unboundedtoeplitz}. In particular, \cite{sarason2008unboundedtoeplitz} appealed to the Smirnov class, $N^+ := \{ b/a : b, a \in H^\infty \text{ and } a \text{ outer} \},$ and the inner outer factorization of functions in $H^2(\mathbb{D})$ to obtain the following theorem.

\begin{theorem}[Sarason \cite{sarason2008unboundedtoeplitz}]\label{thm:sarason}
Let $f : \mathbb{D} \to \mathbb{C}$ be the symbol for the multiplication operator, $M_f : \mathcal{D}(M_{f}) \to H^2(\mathbb{D})$ with $\mathcal{D}(M_{f}) := \{ g \in H : fg \in H \}$. $\mathcal{D}(M_f)$ is dense in $H^2(\mathbb{D})$ iff $f \in N^+$ given as $f = b/a$ where $b,a \in H^\infty$ satisfy $|a|^2 + |b|^2 = 1$, $a$ outer and $\mathcal{D}(M_f) = aH^2(\mathbb{D})$.
\end{theorem}

The reverse implication of Theorem \ref{thm:sarason} is relatively trivial and exploits the fact that a function $a \in H^\infty$ is outer iff $aH^2(\mathbb{D})$ is dense in $H^2(\mathbb{D})$. The approach to determining which symbols yield densely defined Liouville operators leverages outer functions while also combining the antiderivative operator, $J : H^2(\mathbb{D}) \to H^2(\mathbb{D})$, given as $Jh(z) := \int_0^z h(w) dw = \int_0^1 h(tz) dt.$

\begin{lemma}\label{lem:dense-J}
$\mathbb{C}+JaH^2$ is dense in $H^2$
\end{lemma}

\begin{proof}
Let $h \in H^2$ where $h(z) = \sum_{n=0}^\infty h_n z^n$ and $\sum_{n=0}^\infty |h_n|^2 < \infty$. Then $Jh(z) = \sum_{n=0}^\infty \frac{h_n}{n+1}z^{n+1},$ and as $\sum_{n=0}^\infty \left|\frac{h_n}{n+1}\right|^2 < \| h\|_{H^2}^2 < \infty$, we have $Jh \in H^2$. That is, $J:H^2 \to H^2$. Let $p$ be a polynomial and $\epsilon > 0$. The function $\frac{d}{dz} p$ is also a polynomial, and thus in $H^2(\mathbb{D})$. Hence, there is a function $a g \in aH^2$ such that $\| \frac{d}{dz} p - ag \|_{H^2} < \epsilon$.  Note that $J$ is a norm decreasing operator, and thus, $\| J p - J(ag) \|_{H^2} < \epsilon$. Finally, the range of $J$ is of co-dimension 1 with the Hardy space, and since polynomials are dense in $H^2(\mathbb{D})$, the above discussion established $JaH^2(\mathbb{D})$ as being dense in $zH^2(\mathbb{D})$. Thus, $\mathbb{C} + JaH^2(\mathbb{D})$ is dense inside of $H^2(\mathbb{D})$.
\end{proof}

\begin{proposition}
If $f \in N^+$ with representation $f=b/a$ as above, then $A_f$ is densely defined, with domain, $\mathcal{D}(A_f),$ containing the dense space $\mathbb{C}+J a H^2,$ where $J h(z) = \int_0^z h(w) dw$.
\end{proposition}

\begin{proof}
From Lemma \ref{lem:dense-J}, $\mathbb{C}+JaH^2 \subset H^2$. Moreover, if $g \in \mathbb{C}+JaH^2$, then $g(z) = c + J(ah)$ for some $h \in H^2$ and $c \in \mathbb{C}$. Observe that $\frac{d}{dz} g(z) = ah$, and $f\frac{d}{dz} g(z) = fah = \frac{b}{a} ah = bh \in H^2.$ Therefore $\mathbb{C}+JaH^2 \subset \mathcal{D}(A_f)$.
\end{proof}

\begin{proposition}
If $f$ is the symbol for a densely defined operator $A_f$, then $f$ is analytic, and $f = \frac{b}{a \frac{d}{dz} \Phi}$ where $b \in H^2$, $a \in H^2$, $a$ outer, and $\Phi$ a function in BMOA.
\end{proposition}

\begin{proof}Suppose that $\mathcal{D}(A_f)$ is dense inside of $H^2$. Select a nonconstant function $g \in \mathcal{D}(A_f)$. The derivative of $g$, $\frac{d}{dz} g$, is analytic, and hence, so is $h(z) := f(z) \frac{d}{dz} g(z)$. Therefore, $f(z) = \frac{h(z)}{\frac{d}{dz}g(z)}$ is analytic wherever the derivative of $g$ is nonvanishing. By density, for each $z_0 \in \mathbb{C}$ there is a corresponding $g$ with nonvanishing derivative at $z_0$ in $\mathcal{D}(A_f)$. Hence, $f$ is analytic on $\mathbb{D}$.
By \cite{cohn1999factorization}, $\frac{d}{dz} g = a \frac{d}{dz} \Phi$ where $a \in H^2$ and outer, and $\Phi$ is BMOA. The theorem follows.
\end{proof}

\section{Adjoints of Liouville Operators in the Hardy Space}

\begin{proposition}\label{prop:derivative-adjoint}Let $g^{[j]}_w(z) := \frac{d^j}{d\bar w^j} \left(\frac{1}{1-\bar w z}\right)$, then for all $j \in \mathbb{N}$, $g^{[j-1]}_w \in \mathcal{D}(A_f^*)$, and
\begin{equation}\label{eq:adjoint-on-kernel}A_f^* g^{[j-1]}_w = \sum_{\ell = 0}^{j-1} \overline{f^{(\ell)}(w)} g^{[j-\ell]}_w.\end{equation}
\end{proposition}

\begin{proof}

Let $w \in \mathbb{D}$ and set \[g^{[j]}_w(z) = \frac{d^j}{d\bar w^j} \left(\frac{1}{1-\bar w z}\right) = \sum_{n=j}^\infty \frac{n!}{(n-j)!} z^n \bar{w}^{n-j},\] then $g^{[j]}_w \in H^2$ and $\langle h, g^{[j]}_w\rangle_{H^2}=h^{(j)}(w)$.

Suppose $h \in \mathcal{D}(A_f)$, and consider \[ \langle A_f h, g^{[j-1]}_w \rangle_{H^2} = \langle h'\cdot f, g^{[j-1]}_w \rangle_{H^2} = \sum_{\ell = 0}^{j-1} h^{(j-\ell)}(w) f^{(\ell)}(w) = \left\langle h, \sum_{\ell = 0}^{j-1} \overline{f^{(\ell)}(w)} g^{[j-\ell]}_w \right\rangle_{H^2}.\] Thus, $g^{[j-1]}_w \in \mathcal{D}(A_f^*)$ for all $j \in \mathbb{N}$ and $w \in \mathbb{D}$, and 
\[ A_f^* g^{[j-1]}_w = \sum_{\ell = 0}^{j-1} \overline{f^{(\ell)}(w)} g^{[j-\ell]}_w. \]
\end{proof}

\begin{proposition}
Suppose that $\gamma:[0,T] \to \mathbb{D}$ is a continuous trajectory satisfying $\frac{d}{dt} \gamma(t) = f(\gamma(t))$, then $A_f^* \Gamma_{\gamma} = k_{\gamma(T)} - k_{\gamma(0)}.$ More generally, suppose that $\theta : [0,T] \to \mathbb{D}$ is also a continuous trajectory, then $A_f^* \Gamma_\theta(z) = \int_0^T \overline{ f(\theta(t)) g^{[1]}_{z}(\theta(t))} dt.$
\end{proposition}

\begin{proof}
Let $g \in \mathcal{D}(A_f)$ and consider 
\begin{gather*}
\langle A_f g, \Gamma_{\gamma} \rangle_{H^2(\mathbb{D})} = \int_{0}^T f(\gamma(t)) g(\gamma(t)) dt\\
= \int_{0}^T \frac{d}{dt} g(\gamma(t)) dt = g(\gamma(T)) - g(\gamma(0)) = \langle g, k_{g(T)} - k_{g(0)} \rangle_{H^2(\mathbb{D})}.
\end{gather*}
Hence, $A_f^* \Gamma_{\gamma} = k_{g(T)} - k_{g(0)}.$

For the application of the adjoint on $\Gamma_{\theta}$, note that the functional \[g \mapsto \langle A_{f}^* g, \Gamma_{\theta} \rangle_{H^2(\mathbb{D})} = \int_0^T f(\theta(t))g'(\theta(t))dt\] is bounded as $f$ is continuous and the image of $\theta$ is compact within $\mathbb{D}$ and $g'(\theta(t)) = \langle g, g^{[1]}_{\theta(t)} \rangle_{H^2(\mathbb{D})}$. Hence, there is a function $\tilde h$ such that \[ \langle A_{f}^* g, \Gamma_{\theta} \rangle_{H^2(\mathbb{D})} = \int_0^T f(\theta(t))g'(\theta(t))dt = \langle g, \tilde h \rangle_{H^2(\mathbb{D})}, \] and $\tilde h = A_f^* \Gamma_{\theta}$. Moreover, \[\tilde h(z) = \langle \tilde h, k_z \rangle_{H^2(\mathbb{D})} = \overline{ \langle k_z, \tilde h \rangle_{H^2(\mathbb{D})}} = \int_0^T \overline{ f(\theta(t)) g^{[1]}_{z}(\theta(t))} dt.\]
\end{proof}

The above propositions establish the action of the adjoint on particular vectors related to kernel functions. Similar results can be worked out for any RKHS. However, the establishment of a general expression for the adjoint of the Liouville operator is much more involved, and a closed form solution is not expected to be able to be found for general RKHSs. The remainder of the section establishes a formula for the adjoint of the Liouville operator over the Hardy space, which nontrivially leverages the Hardy space's connection with $L^2(\mathbb{T})$ through radial limits.

\begin{theorem}\label{adjoint formula}
Let $f$ be the symbol for a densely defined Liouville operator over the Hardy space, and suppose that $h \in \mathcal{D}(A_f^*)$, then 
\[A_f^* h(z)=P_{H^2}\left(\overline{\frac{f(z)}{z}}\frac{d}{dz}(zh(z))-\overline{\frac{df}{dz}}h(z)\right).\]
\end{theorem}

\begin{proof}
Suppose that $g \in \mathcal{D}(A_f)$ and $h \in \mathcal{D}(A_f^*)$, then
\begin{align*}
\ip{A_fg}{h}_{H^2}&=\lim_{r\rightarrow 1}\int_0^{2\pi} f(z)\frac{d}{dz}\left(g(z)\right)\overline{h(z)}\, d\theta\\
&=\lim_{r\rightarrow 1}\int_0^{2\pi} \left(\frac{1}{iz}\frac{dg}{d\theta}\right)f(z)\overline{h(z)}\, d\theta\\
&=\lim_{r\rightarrow 1}\int_0^{2\pi}\frac{dg}{d\theta}\left(\frac{1}{iz}f(z)\overline{h(z)}\right)\, d\theta\\
&=\lim_{r\rightarrow 1}\left[\frac{1}{iz}f(z)g(z)\overline{h(z)}\Bigg\vert^{2\pi}_0\right]-\lim_{r\rightarrow 1}\int_0^{2\pi}g(z)\frac{d}{d\theta}\left[
\frac{1}{iz}f(z)\overline{h(z)}\right]\, d\theta\\
&=-\lim_{r\rightarrow 1}\int_0^{2\pi}g(z)\left[
\frac{1}{iz}f(z)\frac{d\bar{h}}{d\theta}+\overline{h(z)}\frac{d}{d\theta}\left(\frac{f(z)}{iz}\right)\right]\, d\theta\\
&=-\lim_{r\rightarrow 1}\int_0^{2\pi}g(z)\left[
\frac{1}{iz}f(z)\left(-i\bar{z}\overline{\frac{dh}{dz}}\right)+\overline{h(z)}\left(\frac{df}{dz}-\frac{f(z)}{z}\right)\right]\, d\theta\\
&=\lim_{r\rightarrow 1}\int_0^{2\pi}g(z)\left[
\left(\frac{f(z)}{z}\right)\overline{\left(z\frac{dh}{dz}+h(z)\right)}-\overline{h(z)}\frac{df}{dz}\right]\, d\theta\\
&=\ip{g}{\overline{\frac{f(z)}{z}}\frac{d}{dz}(zh(z))-\overline{\frac{df}{dz}}h(z)}_{L^2(\mathbb{T})}\\
&=\ip{P_{H^2}g}{\overline{\frac{f(z)}{z}}\frac{d}{dz}(zh(z))-\overline{\frac{df}{dz}}h(z)}_{L^2(\mathbb{T})}\\
&=\ip{g}{P_{H^2}\left(\overline{\frac{f(z)}{z}}\frac{d}{dz}(zh(z))-\overline{\frac{df}{dz}}h(z)\right)}_{H^2}
\end{align*}
\end{proof}

\section{Spectrum of Liouville Operators}
The connection between Liouville operators and dynamical systems is realized through the eigenfunctions of the Liouville operator, where if $\phi$ is an eigenfunction of $A_f$ with eigenvalue $\lambda$ and $\gamma:[0,T] \to \mathbb{D}$ is a trajectory satisfying $\frac{d}{dt} \gamma(t) = f(\gamma(t))$, then $\frac{d}{dt} \phi(\gamma(t)) = \phi'(\gamma(t)) f(\gamma(t)) = A_f \phi(\gamma(t)) = \lambda \phi(\gamma(t)).$ Hence, $\phi(\gamma(t)) = \phi(\gamma(0)) e^{\lambda t}$ for all $t \in [0,T]$. This connection is leveraged in the study of DMD to provide data driven models of nonlinear dynamical systems \cite{rosenfelddmd}. The application of Liouville operators to the study of DMD for nonlinear dynamical systems motivates the further investigation of the spectrum of these operators.

In general, the spectrum of the Liouville operator is at least dependent on the properties of the symbol $f$. We start with a proposition which ties the spectrum to the existence of zeros in the disc. 
\begin{proposition}
  Let $f\in H^2(\mathbb{D})$ be a function with no zeros in a neighborhood of the closed disc. Then $\sigma(A_f)=\mathbb{C}$
\end{proposition}
\begin{proof}
We will show that $(A_f-\lambda)$ is not an injective operator, i.e. there exists a non-zero $g(z)$ such that $(A_f-\lambda)g(z)=0$. This function is given by 
\[g(z)=C\exp\left(\int_0^z\frac{\lambda}{f(z)}dz\right)\]
where the above is the path integral from zero to $z$. Note, $\frac{\lambda}{f(z)}$ is bounded, hence the integral and ultimately $g(z)$ is bounded. 
\end{proof}

In the next sub-section we show that it is possible to get a spectrum which is \emph{not} the entire plane if the symbol is allowed to have zeros in the disc. 

\subsection{Symbols with zeros in $\mathbb{D}$}
\begin{lemma}\label{lem:Az-symmetric}
The operator $A_z$ is symmetric over $H^2(\mathbb{D})$.
\end{lemma}

\begin{proof}
Let $g(z) = \sum_{n=0}^\infty g_n z^n \in \mathbb{D}(A_z)$ and $h = \sum_{n=0}^\infty h_n z^n\in \mathbb{D}(A_z^*)$. Then 
\begin{gather*}\langle A_z g, h \rangle_{H^2} = \left \langle \sum_{n=1}^\infty n g_n z^n, \sum_{n=0}^\infty h_{n} z^n \right\rangle_{H^2}\\
= \sum_{n=1}^\infty n g_n \overline{h_n} = \sum_{n=1}^\infty  g_n \overline{n h_n}
= \langle g, A_z h \rangle_{H^2}.
\end{gather*}
\end{proof}

\begin{proposition}
If $A_z:D\subset H^2(\mathbb{D})\rightarrow H^2(\mathbb{D})$ is the densely defined Liouville operator given by $A_z(h(z))=zh'(z)$, then $\sigma(A_z)=\mathbb{N}$
\end{proposition}

\begin{proof} Given that the symbol $f(z)=z$ makes the Liouville operator symmetric according to Lemma \ref{lem:Az-symmetric}, we have automatically that $\sigma(A_z)\subseteq \mathbb{R}$. By inspection we see for $n\in \mathbb{N}$ the functions $g_n(z)=z^n$ are eigen-functions with eigenvalues $n$. Moreover, we can see that these are all the points in the spectrum. 
Suppose there exists an $h\in H^2(\mathbb{D})$ such that 
 \[ (A_z - \lambda) h(z) = g(z) \] for a given $g \in H^2.$
Suppose that $h(z) = \sum_{n=0}^\infty h_n z^n$, then \[h'(z) = \sum_{n=0}^\infty h_{n+1}~(n+1)~z^n \quad\text{and}\quad zh(z) = \sum_{n=1}^\infty h_{n} n z^n.\]
If $g(z)=\sum_{n=0}^\infty g_nz^n$
 \[ -\lambda h_0 + \sum_{n=1}^\infty (nh_n - \lambda h_n)z^n = \sum_{n=0}^\infty g_n z^n.\]

 For $n \ge 0$ we have $h_n = \frac{g_n}{n-\lambda}$. Provided that $\lambda \not\in \mathbb{N}$, $\sum_{n=0}^\infty |h_n|^2 < \infty$. Which means that $\lambda \in \rho(A_f)$. 
\end{proof}
Given a non-real $\alpha\in \mathbb{C}$ the symbol $f(z)=\alpha z$ will not give rise to a symmetric operator. The next proposition shows that $f(z)=z$ is not the only symbol such that $\sigma(A_f)\neq \mathbb{C}$.

\begin{proposition}
If $f(z)=\alpha z+\beta$ for $\alpha,\beta \in \mathbb{C}$ with $|\beta/\alpha|<1$, then  $\sigma(A_f)=\{\alpha\cdot n\mid n\in \mathbb{N}\}$
\end{proposition}
\begin{proof}
In this instance we will compute the spectrum for adjoint for $A_f^*$ and note that $\sigma(A_f^*)=\overline{\sigma(A_f)}$. Suppose that $(A_f^*-\bar{\lambda})h(z)=g(z)$ for some $h,g\in H^2$. Suppose $h(z)=\sum_{n=0}^\infty h_nz^n$ and invoke the theorem \ref{adjoint formula}. We get, 
\begin{align}
\nonumber\sum_{n=1}^\infty \bar{\alpha}nh_nz^n+\sum_{n=2}^\infty\bar{\beta}(n-1)h_{n-1}z^n+\sum_{n=1}^\infty h_{n-1}z^n-\sum_{n=0}^\infty\bar{\lambda}h_nz^n&=\sum_{n=0}^\infty
g_nz^n\\
\label{eq:eigenvectorequation}-h_0\lambda +(\bar{\alpha}h_1+\bar{\beta}h_0-\bar{\lambda}h_1)z+\sum_{n=2}^\infty \left[(\bar{\alpha}n-\bar{\lambda})h_n+\bar{\beta}nh_{n-1}\right]z^n&=\sum_{n=0}^\infty
g_nz^n
\end{align}
From the above we get 
\[h_0=\frac{-g_0}{\lambda}, \quad\text{ and }\quad  h_{n}=\frac{g_n-n\bar{\beta}h_{n-1}}{(n\bar{\alpha}-\bar{\lambda})} \quad \text{for } n\geq 1.\]

Assume that $\lambda \neq \alpha n$ for any $n \in \mathbb{N}$. Write $h_n = d_n g_n + e_n h_{n-1}$. The sequence $e_n \to \frac{\bar \beta}{\bar \alpha}$. Hence, $\limsup |e_n| < 1$ and $d_n \to 0$. Without loss of generality, assume that $|e_n| < e < 1$ and $|d_n| < d < 1$ for all $n \ge 1$. Write $E_m := \prod_{n=1}^m e_n$.

Note that $h_1 = d_1 g_1 + e_1 h_0 = d_1 g_1 + e_1 \alpha g_0,$ $h_2 = d_2 g_2 + e_2 d_1 g_1 + E_2 \alpha g_0,$ $h_3 = d_3 g_3 + E_3/E_2 d_2 g_2 + E_3/E_1 d_1 g_1 + E_3 \alpha g_0,$
and and more generally,
\[h_n = d_n g_n + E_n/E_{n-1} d_{n-1} g_{n-1} + \cdots + E_{n}/E_1 d_1 g_1 + E_n/E_0 \alpha g_0.\]

The function $h(z) = \sum_{n=0}^\infty h_n z^n$ may be expressed as the sum of the following terms
\begin{align*}
    -g_0/\lambda    && E_1 \alpha g_0 z  && E_2 \alpha g_0 z^2    && E_3 \alpha g_0 z^3    && E_4 \alpha g_0 z^4    && \cdots\\
                    && d_1 g_1 z         && E_2/E_1 d_1 g_1 z^2   && E_3/E_1 d_1 g_1 z^3   && E_4/E_1 d_1 g_1 z^4   && \cdots\\
                    &&                   && d_2 g_2 z^2           && E_3/E_2 d_2 g_2 z^3   && E_4/E_2 d_2 g_2 z^4   && \cdots\\
                    &&                   &&                       && d_3 g_3 z^3           && E_4/E_3 d_3 g_3 z^4   && \cdots\\
                    &&                   &&                       &&                       && d_4 g_4 z^4           && \cdots.
\end{align*}
Define each sum along the diagonals as $z^iG_i(z)$. The norm of $G_i(z)$ is bounded above by the norm of $g$. Moreover, $\| G_i \| \le e^i \max(d,\alpha) \|g\|$. Hence, $h(z) = \sum_{i=0}^\infty z^i G_i(z)$, and $\| h \| \le \|g\| \max(d,\alpha,1/\lambda) \sum_{i=0}^\infty e^i = \| g\| \frac{\max(d,\alpha,1/\lambda)}{1-e} < \infty$. Thus, $h \in H^2$ as absolutely convergent series converge in Banach spaces.

Thus, as long $\lambda \neq \alpha n$ for some $n$, $A^*_{\alpha z + \beta} - \bar \lambda$ is invertible. If there exists $n \in \mathbb{N}$ such that $\lambda = \alpha n$, then the coefficient on $h_n$ in the left hand side of \eqref{eq:eigenvectorequation} is zero. Hence, $h_n$ is unconstrained an $A_{\alpha z + \beta}^* - \lambda$ is not invertible.
\end{proof}

\begin{lemma}
Let $m,j\in \mathbb{N}$ then, 
\[A_{z^m}^*(z^j)=
\left\{\begin{array}{llr}
0 & \ & j=0, \ldots, m-1\\
\ & &\\
(j-m+1)z^{j-m+1}&\ &  j\geq m
\end{array}
\right.\]
\end{lemma}
\begin{proof}
Apply the adjoint formula.
\end{proof}
\begin{proposition}
For each $m > 1$, all $\lambda \in \mathbb{C}$ is an eigenvalue for $A_{z^m}^*$ with an eigenspace of dimension $m-1$.
\end{proposition}
\begin{proof}
We will produce the eigenvectors using series methods. Applying the adjoint formula from above we have that for $f(z)=z^m$ and $h(z)=\sum_{n=0}^\infty h_nz^n$
we have that 
\[A_f^*h(z)=\sum_{n=0}^\infty nh_{n+m-1}z^n\]
Assume, $\lambda \neq 0$ and apply the eigenvector equation to $h(z)$. We get that \[\lambda h_n=nh_{n+m-1}\]
Thus, for $k\in \{1, \ldots, m-1\}$ and $j\in \mathbb{N}$

\[h(z)=\sum_{k=1}^{m-1}h_k\cdot\left[\sum_{n=0}^\infty \frac{\lambda^n}{\prod_{j=0}^{n-1}(k+j(m-1))}z^{k+n(m-1)}\right]\]
For $k\in \{1,\ldots, m-1\}$ and $m\in\mathbb{N}$ define 
\[H_k(z)=\sum_{n=0}^\infty \frac{\lambda^n}{\prod_{j=0}^{n-1}(k+j(m-1))}z^{k+n(m-1)}.\]
We now show that these are the eigen-functions for $A_f^*$ for $f(z)=z^m$. 
Let $\lambda\in \mathbb{C}$ and $k\in \{1,\ldots, m-1\}$ be fixed, and 
\[a_n=\frac{\lambda^{2n}}{\left[\prod_{j=0}^{n-1}(k+j(m-1))\right]^2}.\]
Note, 
\[\lim_{n\rightarrow \infty}\left|\frac{a_n}{a_{n+1}}\right|=\lim_{n\rightarrow \infty}\left|\frac{\lambda^2}{(k+n(m-1))^2}\right|=0.\] Since $\sum|a_n|^2<\infty$ we have that $\|H_k\|=\sum|a_n|^2\infty$ and $H_k\in H^2(\mathbb{D}).$ 
Since monomials are in the domain of the adjoint $A_{z^m}^*$ we have the following whenever $H_k(z)$ is in the domain of $A_{z^m}^*$. 
\begin{align*}
   A_{z^m}^*(H_k)(z)&=\sum_{n=0}^\infty \frac{\lambda^n}{\prod_{j=0}^{n-1}(k+j(m-1))}A_{z^m}^*\left(z^{k+n(m-1)}\right)\\
   &=0+\sum_{n=1}^\infty \frac{\lambda^n}{\prod_{j=0}^{n-1}(k+j(m-1))}A_{z^m}^*\left(z^{k+n(m-1)}\right)\\
   &=\sum_{n=1}^\infty \frac{\lambda^n}{\prod_{j=0}^{n-1}(k+j(m-1))}\cdot(k+n(m-1)-m+1)\cdot\left(z^{k+n(m-1)-m+1}\right)\\
   &=\sum_{n=1}^\infty \frac{\lambda^n}{\prod_{j=0}^{n-1}(k+j(m-1))}\cdot(k+(n-1)(m-1))\cdot\left(z^{k+(n-1)(m-1)}\right)\\
    &=\sum_{n=1}^\infty \frac{\lambda\cdot\lambda^{n-1}}{\prod_{j=0}^{n-2}(k+j(m-1))}\left(z^{k+(n-1)(m-1)}\right)\\
    &=\lambda\cdot \sum_{n=0}^\infty \frac{\lambda^n}{\prod_{j=0}^{n-1}(k+j(m-1))}z^{k+n(m-1)}\\
    &=\lambda H_k(z)
\end{align*}
The above calculation is valid if and only $H_k$ lands back in $H^2$ i.e. $\lambda\in \mathbb{C}$. 
\end{proof}

\begin{corollary}Let $f$ be the symbol for a densely defined Liouville operator, $A_f$, and suppose $\{ z_j \}_{j \in \Lambda}$ (where $\Lambda$ is a potentially infinite index set) be the collection of zeros of $f$ with multiplicities $\{ m_j \}_{j \in \Lambda}$. Then $0$ is an eigenvalue for $A_f^*$ with eigendimension $\sum_{j \in \Lambda} m_j$.
\end{corollary}
\begin{proof}Since $A_f$ is densely defined, $f$ is analytic in the disc. Then for each $i \in \Lambda$, $f(z) = (z-z_i)^{m_i}f_{i}(z)$, where $f_i(z)$ does not vanish at $z_i$. Note that $f^{(m)}(z_i) = 0$ for $m=1,\ldots,m_i-1$.

By \eqref{eq:adjoint-on-kernel}, \[ A_f^* g^{[j-1]}_{z_i} = \sum_{\ell = 0}^{j-1} \overline{f^{(\ell)}(z_i)} g^{[j-\ell]}_{z_i} = 0\] for $j = 1,\ldots,{m_i}.$ Hence, $g^{[j-1]}_{z_i}$ is an eigenfunction for $A_f^*$ with eigenvalue $0$, and there is a contribution of $m_i$ dimensions to the zero eigenspace. As the above argument applies for all $i \in \Lambda$, the conclusion follows.
\end{proof}

\section{An Application of the Adjoint Formula}
We will show that the only symbol that can give rise to a self-adjoint Liouville operator is $f(z)=cz$ for $c\in \mathbb{R}$. This is more restrictive than what is seen in for Toeplitz operator which requires a real-valued symbol. 

\begin{theorem}If $A_f$ is self-adjoint then $f(z)=cz$ for $c\in \mathbb{R}$. 
\end{theorem}
\begin{proof} Let $f(z)=\sum_{n=0}^\infty f_n z^n$ be the symbol of a Liouville operator. Suppose $A_f$ is self adjoint. Since, $A_f^*$ is closed, the kernel is in the domain of $A_f^*$ and hence $A_f$. If we apply the operators to the kernel function $K_w(z)$ we have that 
\[A_f[K_w](z)=\bar{w}f_0+(\bar{w}f_1+2f_0\bar{w}^2)z+\ldots\]
and 
\[A_f^*[K_w](z)= 0+\langle f, k_{\bar{w}}\rangle z+ \ldots \]
Moreover, the above holds for all $w\in \mathbb{D}$. Comparing term by term we have that $f_0=0$. Hence, we have that $f(w)=f_1w$ by comparing the first two terms. Additionally, since self-adjoint operators must have real spectrum, we note that $f_1\in \mathbb{R}$ since the spectrum is given by $\sigma(A_f)=\{f_1n\mid n\in \mathbb{N}\}$ by our above proposition. 
\end{proof}
\section{Weighted Liouville Operators}
In this section, we'll discusso weighted Liouville Operators and prove some results in analogy to whats know in weighted composition operators. For a good overview of weighted composition operators over the Hardy space see \cite{matache2008weighted}.

\begin{definition}
A weighted Liouville operator, with symbols $f$ and $\varphi$, is given formally as  \[A_{f,\varphi}g(z)=f(z)\frac{d}{dz}(g(\varphi(z)))=f(z)\varphi'(z)g'(\varphi(z)).\]
\end{definition}

\subsection{Conditions for Self-Adjointness of Liouville Weighted Composition Operators}  Here we wish to show how the selection of composition symbol, $\varphi(z)$, can influence the form of the symbol $f(z)$. Specifically, assuming $A_{f,\varphi}$ is a bounded and self-adjoint operator, then $A_{f,\varphi} K(\cdot,z) = A_{f,\varphi}^* K(\cdot,z)$, which can be utilized to extract conditions connecting $f$ and $\varphi$. 

\begin{definition}
Define, 
\[K^{(1)}_w(z)=\frac{z}{(1-\bar{w}z)^2}.\]
We call $K^{(1)}_w(z)\in H^2$ the reproducing kernel for the derivative of a Hardy space function at the point $w\in \mathbb{D}$.
\end{definition}
\begin{lemma}For a densely defined weighted Liouville operator we have $A_{f,\varphi}K_w=\overline{f(w)\varphi'(w)}K^{(1)}_{\varphi(w)}.$ 

\end{lemma}

\begin{proof}Suppose that $g \in \mathcal{D}(A_{f,\varphi})$. Then by the reproducing property
\[\ip{A_{f,\varphi}g}{K_w}=g'(\varphi(w))\varphi'(w)f(w)=\ip{g}{\overline{f(w)\varphi'(w)}K^{(1)}_{\varphi(w)}}.\]
Therefore, 
\[A_{f,\varphi}^*K_w=\overline{f(w)\varphi'(w)}K^{(1)}_{\varphi(w)}\]
\end{proof}

\begin{theorem}\label{thm:self-adjoint}
If $A_{f,\varphi}$ is bounded, then $A_{f,\varphi}$ is self adjoint then necessarily 
\begin{equation}\label{relationonsymbols}
    \varphi'(z)f(z)=\frac{(z-\overline{\varphi(0)}z^2)(\overline{\varphi'(0)f'(0)+f(0)\varphi''(0)})+2z^2\overline{\varphi'(0)f(0)}}{(1-\overline{\varphi(0)}z)^3}
\end{equation}
is satisfied. 
\end{theorem}

\begin{proof}
If $A_{f,\varphi}^* K(z,\alpha) = A_{f,\varphi} K(z,\alpha)$, then \begin{equation}\label{eq:selfadjointrelation} \frac{\overline{\varphi'(\alpha) f(\alpha)} z}{(1-\overline{\varphi(\alpha)}z)^2} = \frac{\varphi'(z)f(z) \bar\alpha}{(1-\bar \alpha \varphi(z))^2}\end{equation} for the specific case of the Hardy space. This gives the necessary condition for a bounded self adjoint Liouville weighted composition operator over the Hardy space. As $z$ is independent of $\bar{\alpha}$ we can take the derivative with respect to $\bar{\alpha}$:
\begin{equation}
\frac{z(1-\overline{\varphi(\alpha)}z)(\overline{\varphi'(\alpha)f'(\alpha)+f(\alpha)\varphi''(\alpha)})+2z^2\overline{\varphi'(\alpha)f(\alpha)}}{(1-\overline{\varphi(\alpha)}z)^3}-\frac{(1-\bar{\alpha}\varphi(z))\varphi'(z)f(z)+2\bar{\alpha}\varphi(z)}{(1-\bar{\alpha}\varphi(z))^3}
\end{equation}
Setting $\bar{\alpha}=0$ and rearranging gives
\[\varphi'(z)f(z)=\frac{(z-\overline{\varphi(0)}z^2)(\overline{\varphi'(0)f'(0)+f(0)\varphi''(0)})+2z^2\overline{\varphi'(0)f(0)}}{(1-\overline{\varphi(0)}z)^3}\]
\end{proof}

\begin{proposition}\label{prop:occkernelselfadjoint}
Consider the Liouville weighted composition operator, $A_{f,\varphi}$, as in Theorem \ref{thm:self-adjoint}. Let $T > 0$ and suppose that $\gamma:[0,T] \to \mathbb{D}$ satisfies $\dot \gamma(t) = f(\gamma(t))$ for $t \in [0,T]$, and let $\Gamma_{\gamma} \in H^2$ be the occupation kenrel in $H^2$ corresponding to $\gamma$. Then the following relation holds:
\begin{equation}\label{eq:occkernel_selfadjoint}
\left[ \int_0^T K'_{\gamma(t)}(z)dt \right] \phi'(z)f(z) = K_{\phi(\gamma(T))} - K_{\phi(\gamma(0))}.
\end{equation}
\end{proposition}

\begin{proof}
The relation in \eqref{eq:occkernel_selfadjoint} follows from taking the derivative under the integral of $\Gamma_{\gamma}(z) = \int_0^T K_{\gamma(t)}(z) dt$ on the left hand side, and leveraging Theorem \ref{thm:occupationkernel_weightedcomp} on the right.
\end{proof}

It can be noted that \eqref{eq:selfadjointrelation} follows from Proposition \ref{prop:occkernelselfadjoint} after taking the derivative with respect to $T$ of an appropriate trajectory.

\subsection{Boundedness for Liouville Weighted Composition Operators }The action of the adjoint of a weighted composition operator on a kernel function gives some immediate bound conditions. One can establish that for a weighted composition operator we have $W^*_{f,\varphi}K_w=f(\omega)K_{\varphi(\omega)}$. 
As a corrollary, if $W_{f,\varphi}$ is a bounded weighted composition operator on $H^2(\mathbb{D})$ then necessarily 
$B:=\sup\left\{\frac{|f(\omega)|^2(1-|\omega|^2)}{1-|\varphi(\omega)|^2}\ : \ \omega\in \mathbb{D}\right\}<\infty$.
We will prove an analogous proposition.

\begin{proposition}\label{conj1}
If $A_{f,\varphi}$ is bounded then necessarily 
\[B':=\sup\left\{\frac{|f(\omega)|^2|\varphi'(\omega)|^2(1-|\omega|^2)}{(1-|\varphi(\omega)|^2)^2}\cdot\left(\frac{1+|\varphi(w)|^2}{1-|\varphi(w)|^2}\right)\ : \ \omega\in \mathbb{D}\right\}<\infty \]
\end{proposition}
\begin{proof}
Let $k_w=\sqrt{1-|w|^2}K_w$ be the normalized kernel at $w\in \mathbb{D}$. Note, 
\begin{equation}\label{action on kernel}
    \|A_{f,\varphi}^*k_w\|^2=|f(w)||\varphi'(w)|^2\left(1-|w|^2\right)\|K^{(1)}_{\varphi(w)}\|^2
\end{equation}
where 
\[\|K^{(1)}_{\varphi(w)}\|^2=\frac{1+|\varphi(w)|^2}{\left(1-|\varphi(w)|^2\right)^3}.\]
The proposition is a direct consequence of the above formulas and the fact that \[\|A_{f,\varphi}^*k_w\|\leq \|A_{f,\varphi}^*\|^2=\|A_{f,\varphi}\|^2.\]
\end{proof}
If one has the additional assumption that $\varphi$ is a finite Blashcke product then a stronger theorem can stated. In particular, the additional structure given by the appearance of the derivative of $\varphi$ leads to the following corollary. 
\begin{corollary}
If $\varphi$ is a finite Blashcke product and $A_{f,\varphi}$ is bounded then necessarily 
\[B':=\sup\left\{\frac{|f(\omega)|^2(1+|\varphi(\omega)|^2)}{(1-|\varphi(\omega)|^2)^2}\ : \ \omega\in \mathbb{D}\right\}<\infty \]
\end{corollary}
\begin{proof}
We need only note that 
\[\lim_{|\omega|\rightarrow 1}\frac{\varphi'(\omega)(1-|\omega|^2)}{1-|\varphi(\omega)|^2}=1\]
when $\varphi$ is a finite Blashcke product as noted in \cite{Ross-Garcia-Masreghi}
\end{proof}

\paragraph{Compactness for Liouville Weighted Composition Operators}When concerned with the compactness of Liouville weighted composition operators again some amount of insight can be gained immediately. 

\begin{proposition}If $A_{f,\varphi}$ is a compact Liouville weighted composition operator then necessarily, 
\[\lim_{|w|\rightarrow  1^-}\frac{|f(\omega)|^2|\varphi'(\omega)|^2(1-|\omega|^2)}{(1-|\varphi(\omega)|^2)^2}\cdot\left(\frac{1+|\varphi(w)|^2}{1-|\varphi(w)|^2}\right)=0\]
\end{proposition}
\begin{proof}
Since $\|A^*_{f,\varphi}k_w\|\rightarrow 0$ as $|w|\rightarrow 1^-1$ from Equation \ref{action on kernel} which is equivalent to the above expression.
\end{proof}
\begin{proposition}\label{liminf}
If $A_{f,\varphi}$ is compact with $f\in H^2(\mathbb{D})$ and not identically zero then necessarily 
\[\limsup_n \{n^2|\varphi(z)|^{2n-1}|\varphi'(z)|^2\}<1\]
almost everywhere on $\mathbb{T}$.
\end{proposition}
\begin{proof}

For the sake of contradicition assume there existed a set $E\subset \mathbb{T}$ with non-zero measure where the above condition does not hold since compact operators take weakly converging sequences to norm convergent sequences we get a contradiction on the fact that $z^n$ converges weakly to zero but 
\[\|A_{f,\varphi}(z^n)\|_2^2=\int_{\mathbb{T}}|f(z)|^2 n^2|\varphi(z)|^{2n-1}|\varphi'(z)|^2\,dz\geq\int_{E}|f(z)|^2\,dz\]
and this holds for any function $f$. If $f$ is not identically zero we get a contradiction.    
\end{proof}

The above approach of computing the $L^2(\mathbb{T})$ norm for the monomials can be pushed further. Namely, one can compute the Hilbert Schmidt norm. When finite, this implies the operator is compact and thus gives a sufficient condition for compactness. 

\begin{proposition}
$A_{f,\varphi}$ is Hilbert-Schmidt if and only if
\[-\int_{\mathbb{T}}|f(z)|^2|\varphi'(z)|^2\frac{|\varphi(z)|^2(|\varphi(z)|^2+1)}{(|\varphi(z)|^2-1)^3}\, dz<\infty\]
\end{proposition}
\begin{proof}
Since $\varphi(z)$ and its derivative are analytic then in order for the quantity appearing in Proposition \ref{liminf} to be bounded we need $|\varphi|<1$ on $\mathbb{T}$. Hence, using the standard orthonormal basis for $H^2(\mathbb{D})$
\[\|A_{f,\varphi}\|_{\text{HS}}=\sum_{n=0}^\infty \int_{\mathbb{T}}|f(z)|^2 n^2|\varphi(z)|^{2n-1}|\varphi'(z)|^2\, dz = -\int_{\mathbb{T}}|f(z)|^2|\varphi'(z)|^2\frac{|\varphi(z)|^2(|\varphi(z)|^2+1)}{(|\varphi(z)|^2-1)^3}\, dz.\]
\end{proof}

\bibliographystyle{plain}
\bibliography{references}

\end{document}